\documentclass[sn-mathphys,Numbered]{sn-jnl}

\usepackage{graphicx}%
\usepackage{multirow}%
\usepackage{amsmath,amssymb,amsfonts}%
\usepackage{amsthm}%
\usepackage{mathrsfs}%
\usepackage[title]{appendix}%
\usepackage{xcolor}%
\usepackage{textcomp}%
\usepackage{manyfoot}%
\usepackage{booktabs}%
\usepackage{algorithm}%
\usepackage{algorithmicx}%
\usepackage{algpseudocode}%
\usepackage{listings}%
\usepackage{lmodern}

 \newtheorem{thm}{Theorem}[section]
 \newtheorem{cor}[thm]{Corollary}
 
 \newtheorem{prop}[thm]{Proposition}
 \theoremstyle{definition}
 \newtheorem{defn}[thm]{Definition}
 \theoremstyle{remark}
 \newtheorem{rem}[thm]{Remark}
 \newtheorem{ex}[thm]{Example}
 \numberwithin{equation}{section}

\newcommand{\NN}{\mathbb{N}}

\begin{document}

\title[Three point analogue of  \'Ciri\'c-Reich-Rus type mappings]{Three point analogue of  \'Ciri\'c-Reich-Rus type mappings with non-unique fixed points}


\author[1]{\fnm{Ravindra K.} \sur{Bisht}}\email{ravindra.bisht@yahoo.com}

\author*[2]{\fnm{Evgeniy} \sur{Petrov}}\email{eugeniy.petrov@gmail.com}

\affil[1]{
\orgdiv{Department of Mathematics},
\orgname{National Defence Academy},
\orgaddress{
\city{Pune},
\state{Khadakwasla},
\country{India}}}

\affil*[2]{
\orgdiv{Function Theory Department},
\orgname{Institute of Applied Mathematics and Mechanics of the NAS of Ukraine}, \orgaddress{\street{Batiuka str. 19},
\city{Slovyansk},
\postcode{84116},
\country{Ukraine}}}



\abstract{In this paper, we introduce a three-point analogue of \'Ciri\'c-Reich-Rus type mappings, termed as generalized \'Ciri\'c-Reich-Rus type mappings. We demonstrate that these mappings generally exhibit discontinuity within their domain of definition but necessitate continuity at their fixed points. We showcase the existence and non-uniqueness of fixed points for these generalized \'Ciri\'c-Reich-Rus type mappings. By imposing additional conditions, specifically asymptotic regularity and continuity, we extend the applicability of fixed-point theorems to a broader class of mappings. Finally, we obtain two fixed point theorems for generalized \'Ciri\'c-Reich-Rus type mappings in metric spaces that are not necessarily complete.}




\maketitle

\section{Introduction}

In 1971, independently, \'Ciri\'c~\cite{CB71}, Reich~\cite{Re71} and Rus~\cite{Ru71} extended the Kannan fixed point theorem to cover a broader class of mappings.  Essentially, Reich ~\cite{Re71} established the following result which gives the fixed point for a self discontinuous mapping: Let $T\colon X\to X$ be a mapping on a complete metric space $(X,d)$ with
  \begin{equation}\label{e0}
   d(Tx,Ty)\leqslant a d(x,y)+ bd(x,Tx)+ cd(y,Ty)
  \end{equation}
where $a,b, c\geq 0, a+b+c<1$ and $x,y \in X$. Then $T$ has a unique fixed point.

In what follows, we will refer to the mapping (\ref{e0}) as the \'Ciri\'c-Reich-Rus (\'CRR) mapping. This theorem integrates principles from both the Banach contraction principle (by choosing $b=c=0$) and the Kannan fixed point theorem \cite{Ka68,Re712} with $a=0, b=c=\lambda$. This synthesis provides a unified framework that encompasses and extends the key ideas from these established theorems.

In fixed point theory, major generalizations of fixed point results involve weakening contractivity~\cite{Bo12,Ki03,Ci74,Bw69,MK69,Ra62, Re71,Re72,Rh77,Rh88,Wa12,BPR23,Su18,Pr20,Po21,Be04,Be07}; relaxing topology assumptions or metric spaces ~\cite{CJM21,DBC12,Ba00,KKR90,SRR09,Tu12,JKR11,Co59,KI19}, and extending theorems to deal with multi-valued mappings \cite{AJS18,Na69,L2010,NS10,DD11}. These adaptations broaden the applicability of fixed point theorems across diverse mathematical contexts.

In~\cite{P23}, Petrov initiated a novel class of self-mappings characterized by the contraction of perimeters of triangles.
\begin{defn}[\!\cite{P23}]\label{d0}
Let $(X,d)$ be a metric space with $|X|\geqslant 3$. We say that $T\colon X\to X$ is a \emph{mapping contracting perimeters of triangles} on $X$ if there exists $\alpha\in [0,1)$ such that the inequality
  \begin{equation}\label{mcpt}
   d(Tx,Ty)+d(Ty,Tz)+d(Tx,Tz) \leqslant \alpha (d(x,y)+d(y,z)+d(x,z))
  \end{equation}
  holds for all three pairwise distinct points $x,y,z \in X$.
\end{defn}

In addition to the distinctive feature of mapping three points instead of the conventional two, a crucial condition was emphasized to prevent the mapping from having points with the least period two. Furthermore, instead of ensuring the uniqueness of the fixed point, it guarantees the existence of exactly two fixed points. The Banach contraction principle significantly reduces to a noteworthy subclass within this category of mappings.

\begin{rem}
It is crucial to note that the requirement for $x, y, z \in X$ to be pairwise distinct is essential in Definition \ref{d0}. Without this condition, one can observe that the definition becomes equivalent to that of a contraction mapping.
\end{rem}

In \cite{PB23}, the authors introduced a three point analogue of Kannan type mappings, specifically referred to as the generalized Kannan type mappings, and established several fixed point theorems under various conditions. It is crucial to emphasize that these generalized Kannan type mappings are independent from the conventional Kannan type mappings.

\begin{defn}[\!\cite{PB23}]\label{d1}
Let $(X,d)$ be a metric space with $|X|\geqslant 3$. We say that $T\colon X\to X$ is a \emph{generalized Kannan type mapping} on $X$ if there exists $\lambda\in [0,\frac{2}{3})$ such that the inequality
  \begin{equation}\label{e1}
   d(Tx,Ty)+d(Ty,Tz)+d(Tx,Tz) \leqslant \lambda (d(x,Tx)+d(y,Ty)+d(z,Tz))
  \end{equation}
  holds for all three pairwise distinct points $x,y,z \in X$.
\end{defn}

In the next definition, we introduce three-point analogue of the \'CRR type mappings.
\begin{defn}\label{d2}
Let $(X,d)$ be a metric space with $|X|\geqslant 3$. We shall say that $T\colon X\to X$ is a \emph{generalized \'CRR type mapping} on $X$ if there exists $\alpha, \lambda \geq 0, 2\alpha+\frac{3\lambda}{2}< 1$ such that the inequality
  \begin{equation}\label{e00}
  \begin{aligned}
    &d(Tx,Ty) + d(Ty,Tz) + d(Tx,Tz) \\
    &\leqslant \alpha (d(x,y) + d(y,z) + d(z,x)) \\
    &\quad + \lambda (d(x,Tx) + d(y,Ty) + d(z,Tz))
  \end{aligned}
\end{equation}
holds for all three pairwise distinct points $x,y,z \in X$.
\end{defn}

The fixed-point theorem for such mappings is built upon the proof strategy essentially used in \'CRR's fixed point theorem. However, a fundamental difference lies in the definition of these mappings, which involves the mapping of three points in the space, as opposed to the conventional two-point mapping. Furthermore, a condition is introduced to preclude these mappings from having periodic points of prime period 2.

In Section~\ref{sec2}, we study connection between generalized \'CRR  type mappings and \'CRR type mappings. Additionally, we provide an example of a discontinuous generalized \'CRR type mapping.

In Section~\ref{sec3}, we prove the main result of this paper Theorem~\ref{t1}, which is a fixed point theorem for generalized \'CRR  type mappings. It is noteworthy that this theorem asserts that the number of fixed points is at most two. Furthermore, it is demonstrated that the generalized \'CRR type mappings are continuous at fixed points.

In Section~\ref{sec4}, we explore fixed point theorems for asymptotically regular generalized \'CRR-type mappings. The condition of asymptotic regularity facilitates an extension of the requirement for the positive parameter $2\alpha + \frac{3\lambda}{2} < 1$ in~(\ref{e00}) to $\alpha \in [0, \frac{1}{2})$ and $\lambda \in [0,1)$, as demonstrated in Theorem~\ref{gkt1m}. The additional condition of continuity enables the derivation of a fixed point theorem applicable to the class of generalized $F$-\'CRR-type mappings, as presented in Theorem~\ref{gkt2m}. Furthermore, we extend the permissible values of the parameter $\lambda$ to the set $[0, \infty)$, as outlined in Corollary~\ref{cgkt2m}.

In Section~\ref{sec5}, similar to the results proved in~\cite{Ka69}, we obtain two other fixed point theorems for generalized \'CRR type mappings. In the first scenario, we relax the requirement for the completeness of the metric space $X$. We assume that $T\colon X\to X$ is continuous at a specific point $x^*\in X$, and there exists a point $x_0\in X$ such that the sequence of iterates $x_n=Tx_{n-1}$, $n=1,2,...$, possesses a sub-sequence $x_{n_k}$ converging to $x^*$, as stated in Theorem~\ref{t2}. In the second scenario, we suppose the additional condition that the mapping $T$ is continuous throughout the entire space, not just at the point $x_0$. Furthermore, condition~(\ref{e1}) is required only on an everywhere dense subset of the space, as presented in Theorem~\ref{t3}.

\section{Some properties of generalized \'CRR type mappings}\label{sec2}

In the next result, we delve into the connections between generalized \'CRR-type mappings and \'CRR-type mappings.

\begin{prop}\label{p21}
\'CRR type mappings with $2a+\frac{3}{2}(b+c)<1$ are generalized \'CRR type mappings.
\end{prop}
\begin{proof}
Let $(X,d)$ be a metric space with $|X|\geqslant 3$, $T\colon X\to X$ be a \'CRR type mapping with $2a+\frac{3}{2}(b+c)<1$. First of all note that the last inequality implies $a+b+c<1$, i.e., $T$ is indeed \'CRR type mapping. Let $x,y,z\in X$ be pairwise distinct. Consider inequality~(\ref{e0}) for the pairs $y, z$, and $z, x$ and:
  \begin{equation}\label{e02}
   d(Ty,Tz)\leqslant  ad(y,z) +bd(y,Ty)+cd(z,Tz),
  \end{equation}
  \begin{equation}\label{e01}
   d(Tz,Tx)\leqslant  ad(z,x)+bd(z,Tz)+cd(x,Tx).
  \end{equation}
Summarizing the left and the right parts of inequalities~(\ref{e0}), ~(\ref{e02}) and~(\ref{e01}) we obtain
\begin{equation*}
\begin{aligned}
   &d(Tx,Ty) + d(Ty,Tz) + d(Tx,Tz) \\
   &\leqslant a(d(x,y) + d(y,z) + d(z,x)) + (b+c)(d(x,Tx) + d(y,Ty) + d(z,Tz)).
\end{aligned}
\end{equation*}
Hence, we get the desired assertion.
\end{proof}

\begin{ex}
Let us construct an example of discontinuous generalized \'CRR type mapping. Let $X=[0,1]$,  $d$ be the Euclidean distance on $X$, $T\colon X\to X$ be a discontinuous mapping such that $Tx=\frac{x}{4}$ for all $x\in[0,1)$ and $T(1)=\frac{1}{8}$. Let us show first that $T$ is a \'CRR type mapping with the coefficients $a=\frac{1}{4}$, $b=c=\frac{1}{8}$. It is easy to see that~(\ref{e0}) holds for all $x,y\in [0,1)$. Without loss of generality consider that $x\in [0,1)$ and $y=1$. Consider~(\ref{e0}) for such $x$ and $y$:
$$
\left| \frac{x}{4}-\frac{1}{8}\right|\leqslant
\frac{1}{4}\left(1-x\right)+\frac{1}{8}\left(x-\frac{x}{4}\right)+\frac{1}{8}\left(1-\frac{1}{8}\right).
$$
Multiplying both parts on $64$ and simplifying, we get
$$
|16x-8|\leqslant 23-10x.
$$
If $x\in [0,\frac{1}{2})$, then we get $8-16x\leqslant 23-10x$ or $0\leqslant 6x+15$, which clearly holds. If $x\in [\frac{1}{2},1)$, then we get $16x-8\leqslant 23-10x$ or $26x\leqslant 31$, which clearly holds for such $x$. Thus,~(\ref{e0}) holds for all $x,y\in [0,1]$ and $T$ is a \'CRR type mapping. Further,
$$
2a+\frac{3}{2}(b+c)=2\frac{1}{4}+\frac{3}{2}\left(\frac{1}{8}+\frac{1}{8}\right)=\frac{1}{2}+\frac{3}{8}<1.
$$
By Proposition~\ref{p21} $T$ is a generalized \'CRR type mapping.
\end{ex}

\begin{ex}
Let us construct an example of generalized \'CRR type mapping, which is not a \'CRR type mapping. Let $X=\{x,y,z\}$, $d(x,y)=1$, $d(x,z)=10$, $d(y,z)=10$ and let $T\colon X\to X$ be such that $Tx=x$, $Ty=y$, $Tz=x$. It is clear that~(\ref{e00}) holds with the coefficients $\alpha= \frac{1}{3}$ and $\lambda = 0$, but~(\ref{e0}) does not hold for any $a<1$ since $x$ and $y$ are two fixed points of $T$.
\end{ex}

\begin{ex}
Let $X=[0,1]$, $d$ be the Euclidean distance on $X$, $T\colon X\to X$ be such that $Tx=\frac{9x}{10}$ for all $x\in[0,1]$. Let us show that $T$ is a \'CRR type mapping but not a generalized \'CRR type mapping. Clearly~(\ref{e0}) holds with $a = \frac{9}{10}$, $b=c=0$, $a+b+c<1$ for all $x,y \in X$:
$$
\frac{9}{10}|x-y|\leqslant \frac{9}{10}|x-y|+0\left|x-\frac{9}{10}x\right|+0\left|y-\frac{9}{10}y\right|.
$$
Consider~(\ref{e00}):
$$
\frac{9}{10}(|x-y|+|y-z|+|x-z|)\leqslant \alpha (|x-y|+|y-z|+|x-z|)+\lambda \left(\frac{x}{10}+\frac{y}{10}+\frac{z}{10}\right).
$$
It is clear that this inequality does not hold even with $\alpha =\frac{1}{2}$ and $\lambda =\frac{2}{3}$, e.g., for the points $x=0$, $y=\frac{1}{2}$, $z=1$.
\end{ex}

Thus, the previous two examples show that the classes of \'CRR type mappings and generalized \'CRR type mappings are independent.

\begin{prop}\label{p1}
Let $(X,d)$ be a metric space and let  $T\colon X\to X$ be a generalized \'CRR type metric with some  $\alpha, \lambda\geqslant 0$. If $x$ is an accumulation point of  $X$ and $T$ is continuous at $x$, then the inequality
\begin{equation}\label{w1}
d(Tx,Ty)  \leqslant \alpha d(x,y)+ \lambda \left(d(x,Tx) + \frac{d(y,Ty)}{2}\right)
\end{equation}
holds for all points $y\in X$.
\end{prop}

\begin{proof}
Let $x\in X$ be an accumulation point and let $y\in X$. If $y=x$, then clearly~(\ref{w1}) holds. Let now $y\neq x$. Since $x$ is an accumulation point, then there exists a sequence $z_n\to x$ such that $z_n\neq x$, $z_n\neq y$ and all $z_n$ are different.
Hence, by~(\ref{e00}) the inequality
  \begin{equation*}
  \begin{aligned}
    &d(Tx,Ty) + d(Ty,Tz_n) + d(Tx,Tz_n) \\
    &\leqslant \alpha (d(x,y) + d(y,z_n) + d(z_n,x)) \\
    &\quad + \lambda (d(x,Tx) + d(y,Ty) + d(z_n,Tz_n))
  \end{aligned}
\end{equation*}
holds for all $n\in \NN$. Since $z_n\to x$ and  $T$ is continuous at $x$ we have $Tz_n \to Tx$. Since every metric is continuous we have $d(z_n, Tz_n) \to d(x,Tx)$. Hence, we get
\begin{equation*}
d(Tx,Ty) + d(Ty,Tx) \leqslant \alpha (d(x,y) + d(y,x))+ \lambda (d(x,Tx) + d(y,Ty) + d(x,Tx)),
\end{equation*}
which is equivalent to~(\ref{w1}).
\end{proof}

\begin{cor}\label{cor1}
Let $(X,d)$ be a metric space, $T\colon X\to X$ be a continuous generalized \'CRR type mapping with some $\alpha, \lambda \geqslant 0$ and let all points of $X$ are accumulation points. Then $T$ is a \'CRR type mapping with $b=c$.
\end{cor}
\begin{proof}
According to Proposition~\ref{p1}, inequality~(\ref{w1}) holds as well as the inequality
\begin{equation}\label{w2}
d(Tx,Ty)  \leqslant \alpha d(x,y)+ \lambda \left(d(y,Ty) + \frac{d(x,Tx)}{2}\right).
\end{equation}
Summarizing the left and the right parts of ~(\ref{w1}) and ~(\ref{w2}) and dividing both parts of the obtained inequality by $2$ we get
\begin{equation}\label{q3}
d(Tx,Ty)  \leqslant \alpha d(x,y)+ \frac{3\lambda}{4} \left(d(y,Ty) + d(x,Tx)\right).
\end{equation}
Since $2\alpha+\frac{3\lambda}{2}\in [0,1)$, we have $\alpha+\frac{3\lambda}{4}\in [0,\frac{1}{2})$ and $\alpha+\frac{3\lambda}{2}\in [0,1)$. Setting $a=\alpha$, $b=\frac{3\lambda}{4}$ we get $a+2b<1$ and comparing~(\ref{q3}) with~(\ref{e0}) we obtain the desired assertion.
\end{proof}

\section{The main result}\label{sec3}

Let $T$ be a mapping on the metric space $X$. A point $x\in X$ is called a \emph{periodic point of period $n$} if $T^n(x) = x$. The least positive integer $n$ for which $T^n(x) = x$ is called the prime period of $x$, see, e.g.,~\cite[p.~18]{De22}. In particular, the point $x$ is of prime period $2$ if $T(T(x))=x$ and $Tx\neq x$.

The following theorem is the main result of this paper.
\begin{thm}\label{t1}
Let $(X,d)$, $|X|\geqslant 3$, be a complete metric space and let the mapping $T\colon X\to X$ satisfy the following two conditions:
\begin{itemize}
  \item [(i)] $T$ does not possess periodic points of prime period $2$, i.e., $T(T(x))\neq x$ for all $x\in X$ such that $Tx\neq x$.
  \item [(ii)] $T$ is a generalized \'CRR type mapping on $X$.
\end{itemize}
Then $T$ has a fixed point. The number of fixed points is at most two.
\end{thm}

\begin{proof}
Let $x_0\in X$, $Tx_0=x_1$, $Tx_1=x_2$, \ldots, $Tx_n=x_{n+1}$, \ldots. Suppose that $x_n$ is not a fixed point of the mapping $T$ for every $n=0,1,...$.
Since $x_{n-1}$ is not fixed, then $x_{n-1}\neq x_n=Tx_{n-1}$. By condition (i) $x_{n+1}=T(T(x_{n-1}))\neq x_{n-1}$ and by the supposition that $x_{n}$ is not fixed we have $x_n\neq x_{n+1}=Tx_n$. Hence, $x_{n-1}$, $x_n$ and $x_{n+1}$ are pairwise distinct.
Let us set in~(\ref{e00}) $x=x_{n-1}$, $y=x_n$, $z=x_{n+1}$.
Then
$$
d(Tx_{n-1},Tx_n)+d(Tx_n,T_{x_{n+1}})+d(Tx_{n-1}, Tx_{n+1})
$$
$$
\leqslant
\alpha(d(x_{n-1},x_{n})+d(x_n,x_{n+1})+d(x_{n+1}, x_{n-1}))
$$
$$
+\lambda(d(x_{n-1},Tx_{n-1})+d(x_n,Tx_n)+d(x_{n+1}, Tx_{n+1}))
$$
and
$$
d(x_{n},x_{n+1})+d(x_{n+1},x_{n+2})+d(x_{n+2}, x_{n})
$$
$$
\leqslant
\alpha(d(x_{n-1},x_{n})+d(x_n,x_{n+1})+d(x_{n+1}, x_{n-1}))
$$
$$
\lambda(d(x_{n-1},x_{n})+d(x_n,x_{n+1})+d(x_{n+1}, x_{n+2})).
$$

Using the triangle inequality $d(x_{n+1}, x_{n-1})\leqslant d(x_{n+1}, x_{n})+d(x_{n},x_{n-1})$ we get

$$
d(x_{n},x_{n+1})+d(x_{n+1},x_{n+2})+d(x_{n+2}, x_{n})
$$
$$
\leqslant
\alpha(d(x_{n-1},x_{n})+d(x_n,x_{n+1})+ d(x_{n+1}, x_{n})+d(x_{n}, x_{n-1}))
$$
$$
\lambda(d(x_{n-1},x_{n})+d(x_n,x_{n+1})+d(x_{n+1}, x_{n+2})).
$$

Hence,
$$
(1-\lambda)d(x_{n+1},x_{n+2})\leqslant (2\alpha+\lambda)(d(x_{n-1},x_n)+d(x_n,x_{n+1}))-d(x_n,x_{n+1})-d(x_{n+2},x_n).
$$
Keeping in view of the triangle inequality $d(x_{n+1}, x_{n+2})\leqslant d(x_n, x_{n+1})+d(x_{n+2},x_n)$ we get
$$
(1-\lambda)d(x_{n+1},x_{n+2})
\leqslant (2\alpha+\lambda) (d(x_{n-1},x_n)+d(x_n,x_{n+1}))-d(x_{n+1},x_{n+2}).
$$
Further,
$$
(2-\lambda)d(x_{n+1},x_{n+2})\leqslant (2\alpha+\lambda) (d(x_{n-1},x_n)+d(x_n,x_{n+1})),
$$
$$
d(x_{n+1},x_{n+2})\leqslant \frac{2\alpha+\lambda}{2-\lambda} (d(x_{n-1}, x_{n})+d(x_{n}, x_{n+1}))
$$
and
$$
d(x_{n+1},x_{n+2})\leqslant \frac{2(2\alpha+\lambda)}{2-\lambda} \max\{d(x_{n-1}, x_{n}),d(x_{n}, x_{n+1})\}.
$$
Let $\gamma = \frac{2(2\alpha+\lambda)}{2-\lambda}$. Using the relation $(2\alpha+\frac{3 \lambda}{2}) \in [0, 1)$, we get $\gamma \in [0,1)$. Further,
\begin{equation}\label{e2}
d(x_{n+1},x_{n+2})\leqslant \gamma \max\{d(x_{n-1}, x_{n}),d(x_{n}, x_{n+1})\}.
\end{equation}

Set $a_n=d(x_{n-1},x_n)$, $n=1,2,\ldots,$ and let
$a=\max\{a_{1},a_{2}\}$.
Hence and by~(\ref{e2}) we obtain
$$
a_1\leqslant a, \, \,
a_2\leqslant a, \, \,
a_3\leqslant \gamma a, \, \,
a_4\leqslant \gamma a, \, \,
a_5\leqslant \gamma^2 a, \, \,
a_6\leqslant \gamma^2 a, \,\,
a_7\leqslant \gamma^3 a, \,
\ldots .
$$
Since $\gamma <1$, it is clear that the inequalities
$$
a_1\leqslant a, \, \,
a_2\leqslant a, \, \,
a_3\leqslant \gamma^{\frac{1}{2}} a, \, \,
a_4\leqslant \gamma a, \, \,
a_5\leqslant \gamma^{\frac{3}{2}} a, \, \,
a_6\leqslant \gamma^2 a, \,\,
a_7\leqslant \gamma^{\frac{5}{2}} a, \,
\ldots
$$
also hold. That is,
\begin{equation}\label{e9}
a_n\leqslant \gamma^{\frac{n}{2}-1}a
\end{equation}
for $n=3,4,\ldots$.

Let $p\in \mathbb N$, $p\geqslant 2$. By the triangle inequality, for $n\geqslant 3$ we have
$$
d(x_n,\,x_{n+p})\leqslant d(x_{n},\,x_{n+1})+d(x_{n+1},\,x_{n+2})+\ldots+d(x_{n+p-1},\,x_{n+p})
$$
$$
=a_{n+1}+a_{n+2}+\cdots+a_{n+p} \leqslant
a(\gamma^{\frac{n+1}{2}-1}+\gamma^{\frac{n+2}{2}-1}+\cdots
+\gamma^{\frac{n+p}{2}-1})
$$
$$
=a\gamma^{\frac{n+1}{2}-1}(1+\gamma^{\frac{1}{2}}+\cdots
+\gamma^{\frac{p-1}{2}})=a\gamma^{\frac{n-1}{2}}\frac{1-\sqrt{\gamma^p}}{1-\sqrt{\gamma}}.
$$
Since by the supposition $0\leqslant\gamma<1$, then $0\leqslant \sqrt{\gamma^p}<1$ and $d(x_n,\,x_{n+p})\leqslant a\gamma^{\frac{n-1}{2}}\frac{1}{1-\sqrt{\gamma}}$. Hence, $d(x_n,\,x_{n+p})\to 0$ as $n\to \infty$ for every $p>0$. Thus, $\{x_n\}$ is a Cauchy sequence. By the completeness of $(X,d)$, this sequence has a limit $x^*\in X$.

Recall that any three consecutive element of the sequence $(x_n)$ are pairwise distinct. If $x^*\neq x_k$ for all $k\in \{1,2,...\}$, then inequality~(\ref{e00}) holds for the pairwise distinct points  $x^*$, $x_{n-1}$ and $x_n$.
Suppose that there exists the smallest possible $k\in \{1,2,...\}$ such that $x^*=x_k$.  Let $m>k$ be such that $x^*=x_m$. Then the sequence $(x_n)$ is cyclic starting from $k$ and can not be a Cauchy sequence. Hence, the points $x^*$, $x_{n-1}$ and $x_n$ are pairwise distinct at least when $n-1>k$.

Let us prove that $Tx^*=x^*$. If there exists $k\in \{1,2,...\}$ such that $x_k=x^*$. Then suppose that $n-1>k$. By the triangle inequality and by inequality~(\ref{e00}) we have
$$
d(x^*,Tx^*)\leqslant d(x^*,x_{n})+d(x_{n},Tx^*)
=d(x^*,x_{n})+d(Tx_{n-1},Tx^*)
$$
$$
\leqslant d(x^*,x_{n})+d(Tx_{n-1},Tx^*)+d(Tx_{n-1},Tx_{n})+d(Tx_{n},Tx^*)
$$
$$
\leqslant d(x^*,x_{n})+ \alpha (d(x_{n-1},x^*)
+d(x_{n-1}, x_{n})+d(x_n,x^*))
$$
$$
+\lambda(d(x_{n-1},Tx_{n-1})
+d(x_{n},Tx_{n})+d(x^*,Tx^*)).
$$
Hence,
$$
(1-\lambda)d(x^*,Tx^*)\leqslant (1+\alpha) d(x^*,x_{n})+(\alpha+\lambda)d(x_{n-1},x_{n}) + \lambda d(x_{n},x_{n+1})+ \alpha d(x_{n-1},x^*)
$$
and

\begin{align}\label{e50}
    &d(x^*,Tx^*) \leqslant \\
    &\frac{1}{1-\lambda} \bigl( (1+\alpha) d(x^*,x_{n}) + (\alpha+\lambda)d(x_{n-1},x_{n}) + \lambda d(x_{n},x_{n+1})+ \alpha d(x_{n-1},x^*) \bigr).\nonumber
\end{align}
Since all the distances in the right side tend to zero as $n\to \infty$, we obtain $d(x^*,Tx^*)=0$.

Suppose that there exists at least three pairwise distinct fixed points $x$, $y$ and $z$.  Then $Tx=x$, $Ty=y$ and $Tz=z$, which contradicts to~(\ref{e00}).
\end{proof}

\begin{rem}
Suppose that under the supposition of the theorem the mapping $T$ has a fixed point $x^*$ which is a limit of some iteration sequence $x_0, x_1=Tx_0, x_2=Tx_1,\ldots$ such that $x_n\neq x^*$ for all $n=1,2,\ldots$. Then $x^*$ is a unique fixed point.
Indeed, suppose that $T$ has another fixed point $x^{**}\neq x^*$.
It is clear that $x_n\neq x^{**}$ for all $n=1,2,\ldots$. Hence, we have that the points $x^*$, $x^{**}$ and $x_n$ are pairwise distinct for all $n=1,2,\ldots$. Consider inequality~(\ref{e00}) for the points $x^*$, $x^{**}$ and $x_n$:
$$
d(Tx^*,Tx^{**})+d(Tx^*,Tx_{n})+d(Tx^{**},Tx_{n})
$$
$$
\leqslant \alpha(d(x^{*}, x^{**})+d(x^{*},x_n)+d(x^{**},x_n))+\lambda(d(x^*,Tx^{*})+d(x_{n},Tx_{n})+d(x^{**},Tx^{**}))
$$
Further,
$$
d(x^*,x^{**})+d(x^*,x_{n+1})+d(x^{**},x_{n+1})
$$
$$
\leqslant
\alpha(d(x^*,x^{**})+d(x^*,x_{n})+d(x^{**},x_n))+\lambda d(x_n,x_{n+1}).
$$
Taking into consideration that $d(x^*,x_{n})\to 0$, $d(x^*,x_{n+1})\to 0$, $d(x^{**},x_{n+1})\to d(x^{**},x^*)$ and $d(x_n,x_{n+1})\to 0$ and letting $n\to \infty$, we obtain the inequality
$
d(x^*,x^{**})\leqslant \alpha d(x^*,x^{**})
$
or $1\leqslant \alpha$, which leads to contradiction, since by Definition~\ref{d2} the inequality $\alpha \leqslant \frac{1}{2}$ holds.
\end{rem}

\begin{cor}
Setting $\lambda =0$ in~(\ref{e00}) implies that Theorem~\ref{t1} entails Theorem 2.4 from~\cite{P23} for mappings contracting perimeters of triangles with the coefficient $\alpha \in [0, \frac{1}{2})$.

Similarly, setting $\alpha =0$ in~(\ref{e00}) we get that Theorem~\ref{t1} implies exactly Theorem 3.2 from~\cite{PB23} for generalized Kannan type mappings.
\end{cor}

It is well-known that \'CRR-type mappings are continuous at fixed points~\cite{Rh88}. The following proposition demonstrates that generalized \'CRR-type mappings also exhibit this desirable property.

\begin{prop}
Generalized \'CRR type mapping are continuous at fixed points.
\end{prop}

\begin{proof}
Let $(X,d)$ be a metric space with $|X|\geqslant 3$, $T\colon X\to X$ be a \emph{generalized \'CRR type mapping} and $x^*$ be a fixed point of $T$. Let $(x_n)$ be a sequence such that $x_n\to x^*$, $x_{n}\neq x_{n+1}$ and $x_{n}\neq x^*$ for all $n$. Let us show that $Tx_n\to Tx^*$.
By~(\ref{e00}) we have
  \begin{multline*}
   d(Tx^*,Tx_n)+d(Tx_n,Tx_{n+1})+d(Tx_{n+1},Tx^*) \\ \leqslant \alpha(d(x^*,x_n)+d(x_n,x_{n+1})+d(x_{n+1},x^*))+\lambda (d(x^*,Tx^*)+d(x_n,Tx_n)+d(x_{n+1},Tx_{n+1})).
  \end{multline*}
Using  $d(x_n,x_{n+1})\leq d(x_n, x^*)+ d(x^*, x_{n+1})$, we get
 \begin{multline*}
   d(Tx^*,Tx_n)+d(Tx_n,Tx_{n+1})+d(Tx_{n+1},Tx^*) \\ \leqslant 2\alpha(d(x^*,x_n)+d(x_{n+1},x^*))+\lambda (d(x^*,Tx^*)+d(x_n,Tx_n)+d(x_{n+1},Tx_{n+1})).
\end{multline*}

Hence,
  \begin{multline*}
   d(Tx^*,Tx_n)+d(Tx_{n+1},Tx^*)\\ \leqslant 2\alpha(d(x^*,x_n)+d(x_{n+1},x^*))+ \lambda (d(x_n,Tx_n)+d(x_{n+1},Tx_{n+1})).
  \end{multline*}
By the triangle inequality we have
  \begin{multline*}
   d(Tx^*,Tx_n)+d(Tx_{n+1},Tx^*) \\ \leqslant
   2\alpha(d(x^*,x_n)+d(x_{n+1},x^*))+\lambda (d(x_n,x^*)+d(x^*,Tx_n)+d(x_{n+1},x^*)+d(x^*,Tx_{n+1})).
  \end{multline*}
Further,
 $$
   d(Tx^*,Tx_n)+d(Tx_{n+1},Tx^*) \leqslant \frac{2\alpha+\lambda}{1-\lambda}(d(x_n,x^*)+d(x_{n+1},x^*)).
$$
Since $d(x_n,x^*)\to 0$ and $d(x_{n+1},x^*)\to 0$ we have $$d(Tx^*,Tx_n)+d(Tx_{n+1},Tx^*) \to 0$$ and, hence, $d(Tx^*,Tx_n) \to 0$.

Let now $(x_n)$ be a sequence such that $x_n\to x^*$, and
$x_{n}\neq x^*$ for all $n$, but $x_{n} = x_{n+1}$ is possible. Let $(x_{n_k})$ be a subsequence of $(x_n)$ obtained by deleting corresponding repeating elements of $(x_n)$, i.e.,  such that $x_{n_k}\neq x_{n_{k+1}}$ for all $k$. It is clear that $x_{n_k}\to x^*$. As was just proved $Tx_{n_k}\to Tx^*=x^*$. The difference between $Tx_{n_k}$ and $Tx_{n}$ is that $Tx_{n}$ can be obtained from $Tx_{n_k}$ by inserting corresponding repeating consecutive elements. Hence, it is easy to see that $Tx_{n}\to Tx^*$.

Let $(x_n)$ be a sequence such that $x_n=x^*$ for all $n>N$, where $N$ is some natural number. Then, clearly, $Tx_{n}\to Tx^*$.
Let $(x_n)$ now be an arbitrary sequence such that $x_n\to x^*$ but not like in the previous case. Consider a subsequence  $(x_{n_k})$ obtained from $(x_n)$ by deleting elements $x^*$ (if they exist). Clearly, $x_{n_k}\to x^*$. It was just shown that such $Tx_{n_k}\to Tx^*$. Again, we see that $Tx_n$ can be obtained from $Tx_{n_k}$ by inserting in some places elements $Tx^*=x^*$. Again, it is easy to see that $Tx_{n}\to Tx^*$.
\end{proof}

\section{Asymptotic regularity}\label{sec4}

The concept of asymptotic regularity enables an extension of the parameters within the class of mappings for which the fixed-point theorems hold.

\begin{defn}
Let $(X, d)$ be a metric space. A mapping $T\colon X \to X$ satisfying the condition
\begin{equation}\label{ar}
\lim_{n \to \infty} d(T^{n+1}x, T^nx) = 0
\end{equation}
for all $x \in X$ is called asymptotically regular \cite{BP66}.
\end{defn}

\begin{rem}\label{r32}
Let $(X, d)$ be a metric space $T\colon X \to X$ be a self-mapping and let $x_0\in X$,  $T(x_0) = x_1$, $T(x_1) = x_2$, and so on.
If $T$ is asymptotically regular and the sequence $(x_n)$ does not possess a fixed point of $T$, then all the points $x_i$, $i\geqslant 0$, are pairwise distinct. Indeed, otherwise the sequence $(x_n)$ is cyclic starting from some number and condition~(\ref{ar}) does not hold.
\end{rem}

\begin{thm} \label{gkt1m}
Let $(X, d)$ be a complete metric space with $|X| \geqslant 3$ and let the mapping $T\colon X \to X$ be asymptotically regular generalized \'CRR type mapping with the coefficients $\alpha\in [0,\frac{1}{2})$ and $\lambda \in [0,1)$. Then $T$ has a fixed point. The number of fixed points is at most two.
\end{thm}

\begin{proof}
Let $x_0 \in X$, $T(x_0) = x_1$, $T(x_1) = x_2$, and so on. Suppose that $(x_n)$ does not possess a fixed point of $T$. Let us prove that $(x_n)$ is a Cauchy sequence. It is sufficient to show that $d(x_n,x_{n+p})\to 0$ as $n\to \infty$ for all $p>0$.
If $p=1$, then this follows from the definition of asymptotic regularity. Let $p\geqslant 2$. By Remark~\ref{r32} the points $x_n$, $x_{n+p-1}$, $x_{n+p}$ are pairwise distinct. Using repeated triangle inequality, inequality~(\ref{e00}) and asymptotic regularity, we get
\begin{align*}
& {} d(x_n,x_{n+p})\leqslant d(x_n,x_{n+1})+d(x_{n+1},x_{n+p+1})+d(x_{n+p+1},x_{n+p})\\\leqslant & {}
d(x_n,x_{n+1})+d(x_{n+1},x_{n+p+1})+d(x_{n+p+1},x_{n+p})+d(x_{n+1},x_{n+p})\\\leqslant & {}
d(x_n,x_{n+1})+\alpha(d(x_{n},x_{n+p})+d(x_{n+p},x_{n+p-1})+d(x_{n},x_{n+p-1}))\\ & {}
+\lambda(d(x_{n},Tx_{n})+d(x_{n+p},Tx_{n+p})+d(x_{n+p-1},Tx_{n+p-1}))\\ \leqslant & {}
d(x_n,x_{n+1})+\alpha(d(x_{n},x_{n+p})+d(x_{n+p},x_{n+p-1})+d(x_{n},x_{n+p})+d(x_{n+p},x_{n+p-1}))\\ & {}
+\lambda(d(x_{n},Tx_{n})+d(x_{n+p},Tx_{n+p})+d(x_{n+p-1},Tx_{n+p-1})).\\
\end{align*}
Hence
$$d(x_n,x_{n+p})\leq \frac{1+\lambda}{1-2\alpha} d(x_n,x_{n+1})+ \frac{2\alpha+\lambda}{1-2\alpha} d(x_{n+p},x_{n+p-1})$$
$$+ \frac{\lambda}{1-2\alpha} d(x_{n+p},x_{n+p+1}) \to 0
$$
as $n\to \infty$. Thus, $(x_{n})$ is a Cauchy sequence. Since the points $x^*$, $x_n$, $x_{n+1}$ are pairwise distinct for all $n\geqslant 0$, the rest of the proof follows from the proof given in Theorem~\ref{t1}, see~(\ref{e50}).
\end{proof}

Below we show that the assumption of continuity for the mappings $T$ allows us to obtain fixed point theorems for more general classes of mappings than generalized \'CRR type mappings even with the coefficient $\lambda\in [0,1)$.

We now provide a more general version of generalized \'CRR type mappings and introduce the following definitions. First, we define the class $\mathcal{F}$ of functions $F\colon \mathbb{R}^+ \times \mathbb{R}^+ \times \mathbb{R}^+ \rightarrow \mathbb{R}^+$ satisfying the following conditions:

\begin{enumerate}
    \item[(i)] $F(0, 0, 0) = 0$;
    \item[(ii)] $F$ is continuous at $(0, 0, 0)$.
\end{enumerate}

\begin{defn}\label{d3}
Let $(X,d)$ be a metric space with $|X|\geqslant 3$. We shall say that $T\colon X\to X$ is a \emph{generalized $F$-\'CRR type mapping} on $X$ if there exist $\alpha\in [0, \frac{1}{2})$ and $F \in \mathcal{F}$ such that the inequality
  \begin{equation}\label{e3}
\begin{aligned}
   &d(Tx,Ty) + d(Ty,Tz) + d(Tx,Tz) \\
   &\leqslant \alpha(d(x,y) + d(y,z) + d(z,x)) + F(d(x,Tx),d(y,Ty), d(z,Tz))
\end{aligned}
\end{equation}
  holds for all three pairwise distinct points $x,y,z \in X$.
\end{defn}

\begin{thm} \label{gkt2m}
Let $(X, d)$ be a complete metric space with $|X| \geqslant 3$ and let the mapping
$T\colon X \to X$ be a continuous, asymptotically regular generalized $F$-\'CRR type mapping. Then $T$ has a fixed point. The number of fixed points is at most two.
\end{thm}

\begin{proof}
Let $x_0 \in X$, $T(x_0) = x_1$, $T(x_1) = x_2$, and so on. Suppose that $(x_n)$ does not possess a fixed point of $T$. Let us prove that $(x_n)$ is a Cauchy sequence. It is sufficient to show that $d(x_n,x_{n+p})\to 0$ as $n\to \infty$ for all $p>0$.
If $p=1$, then this follows from the definition of asymptotic regularity. Let $p\geqslant 2$. By Remark~\ref{r32} the points $x_n$, $x_{n+p-1}$, $x_{n+p}$ are pairwise distinct. Using repeated triangle inequality, inequality~(\ref{e3}) and asymptotic regularity, we get

\begin{align*}
& {} d(x_n,x_{n+p})\leqslant d(x_n,x_{n+1})+d(x_{n+1},x_{n+p+1})+d(x_{n+p+1},x_{n+p})\\\leqslant & {}
d(x_n,x_{n+1})+d(x_{n+1},x_{n+p+1})+d(x_{n+p+1},x_{n+p})+d(x_{n+1},x_{n+p})\\\leqslant & {}
d(x_n,x_{n+1})+\alpha(d(x_{n},x_{n+p})+d(x_{n+p},x_{n+p-1})+d(x_{n},x_{n+p-1}))\\ & {}
+F(d(x_{n},Tx_{n}),d(x_{n+p},Tx_{n+p}),d(x_{n+p-1},Tx_{n+p-1}))\\ \leqslant & {}
d(x_n,x_{n+1})+\alpha(d(x_{n},x_{n+p})+d(x_{n+p},x_{n+p-1})+d(x_{n},x_{n+p})+d(x_{n+p},x_{n+p-1}))\\ & {}
+F(d(x_{n},x_{n+1}),d(x_{n+p},x_{n+p+1}),d(x_{n+p-1},x_{n+p})).
\end{align*}
Hence
\begin{equation*}
\begin{aligned}
    &d(x_n,x_{n+p}) \\
    &\leq \frac{1}{1-2\alpha} \left[d(x_n,x_{n+1}) + 2\alpha d(x_{n+p},x_{n+p-1}) \right. \\
    &\quad \left. + F(d(x_{n},x_{n+1}),d(x_{n+p},x_{n+p+1}),d(x_{n+p-1},x_{n+p}))\right] \to 0
\end{aligned}
\end{equation*}
as $n\to \infty$. Thus, $(x_{n})$ is a Cauchy sequence.  By the completeness of $(X, d)$, this sequence has a limit $x^* \in X$. Now
$$
d(Tx^*,x^*)\leqslant d(Tx^*,x_n)+d(x_n,x^*)= d(Tx^*,Tx_{n-1})+d(x_n,x^*)
$$
Since $T$ is continuous letting $n\to \infty$, we obtain $Tx^*=x^*$.
The rest of the proof follows from reasoning similar to the last paragraph of the proof of Theorem~\ref{t1}.
\end{proof}

Let $\mathcal{B}$ represent the class of functions $\beta: [0,\infty) \to [0,\infty)$ that meet the following condition: $\limsup\limits_{t\rightarrow 0}\beta(t) <\infty$.

\begin{defn}\label{d4}
Let $(X,d)$ be a metric space with $|X|\geqslant 3$. We shall say that $T\colon X\to X$ is a \emph{generalized $\mathcal{B}$-\'CRR type mapping} on $X$ if there exist $\alpha \in [0, \frac{1}{2})$ and $\beta_1, \beta_2, \beta_3 \in \mathcal{B}$ such that the inequality
\begin{equation}\label{e4}
   \begin{aligned}
      &d(Tx,Ty) + d(Ty,Tz) + d(Tx,Tz) \leq \alpha(d(x,y) + d(y,z) + d(x,z))\\
      &+ \beta_1(d(x,Tx))d(x,Tx) + \beta_2(d(y,Ty))d(y,Ty) + \beta_3(d(z,Tz))d(z,Tz)
   \end{aligned}
\end{equation}
holds for all three pairwise distinct points $x,y,z \in X$.
\end{defn}

\begin{cor}
Let $(X, d)$ be a complete metric space with $|X| \geqslant 3$ and let the mapping
$T\colon X \to X$ be a continuous, asymptotically regular generalized $\mathcal{B}$-\'CRR type mapping. Then $T$ has a fixed point. Moreover, the number of fixed points is at most two.
\end{cor}
\begin{proof}
Set $F(x,y,z)=\beta_1(x)x+\beta_2(x)x+\beta_3(x)$. Then  $F(0,0,0)=0$ and $\lim\limits_{x, y, z \to 0} F(x,y,z)=0$ since  $\limsup\limits_{t\rightarrow 0} \beta_i(t) <\infty$ for $i=1,2,3$. Hence, this assertion follows from Theorem~\ref{gkt2m}.
\end{proof}

By setting $\beta_1(t)=\beta_2(t)=\beta_3(t)= \lambda \geqslant 0$ in (\ref{e4}), we get a generalized \'CRR type mapping with the coefficients  $\alpha \in [0, \frac{1}{2})$ and $\lambda \in [0,\infty)$. Hence, we immediately obtain the following.

\begin{cor}\label{cgkt2m}
Let $(X, d)$ be a complete metric space with $|X| \geqslant 3$ and let the mapping
$T\colon X \to X$ be a continuous asymptotically regular generalized \'CRR type mapping with the coefficients $\alpha \in [0, \frac{1}{2})$ and $\lambda \in [0,\infty)$. Then $T$ has a fixed point. The number of fixed points is at most two.
\end{cor}

\begin{rem}
 It is noteworthy to mention that the continuity requirement for the mapping $T$ in Corollary~\ref{cgkt2m} can be further relaxed. Corollary~\ref{cgkt2m} remains applicable when the continuity assumption on the mapping $T$ is replaced with various weaker continuity notions, such as orbital continuity, $x_0$-orbitally continuity, almost orbitally continuity, weakly orbitally continuity, $T$-orbitally lower semi-continuity, or $k$-continuity. Interested readers can find more information on these weaker continuity notions in~\cite{B23}.
\end{rem}

Let $\mathcal{U}$ denote the class of functions $\varphi \colon [0, \infty) \to [0, \infty)$ satisfying:
\begin{enumerate}
  \item[(a)] $\varphi(t) < t$ for all $t > 0$,
  \item[(b)] $\varphi$ is upper semi-continuous.
\end{enumerate}

\begin{defn}\label{d5}
Let $(X,d)$ be a metric space with $|X|\geqslant 3$. We shall say that $T\colon X\to X$ is a \emph{generalized $(\varphi-F)$-\'CRR type mapping} on $X$ if there exist $\varphi \in \mathcal{U}$ and $F \in \mathcal{F}$ such that the inequality
  \begin{equation}\label{e5}
\begin{aligned}
   &d(Tx,Ty) + d(Ty,Tz) + d(Tx,Tz) \\
   &\leqslant \varphi\max\{d(x,y),d(y,z),d(z,x)\} + F(d(x,Tx),d(y,Ty), d(z,Tz))
\end{aligned}
\end{equation}
  holds for all three pairwise distinct points $x,y,z \in X$.
\end{defn}

\begin{thm} \label{gkt6m}
Let $(X, d)$ be a complete metric space with $|X| \geqslant 3$ and let the mapping
$T\colon X \to X$ be a continuous, asymptotically regular generalized $(\varphi-F)$-\'CRR type mapping. Then $T$ has a fixed point. The number of fixed points is at most two.
\end{thm}

\begin{proof}
Let $x_0 \in X$, $T(x_0) = x_1$, $T(x_1) = x_2$, and so on. Suppose that $(x_n)$ does not possess a fixed point of $T$. Let us prove that $(x_n)$ is a Cauchy sequence. Suppose the converse.
Then there exists $\varepsilon >0$, so that for each $i=1,2,\ldots $ there are some integers $n_i=n_i(i),p_i=p_i(i)\in \mathbb{N}$ such that $n_i\geqslant i$ and
\begin{equation}\label{rr1}
 d(x_{n_i},x_{n_i+p_i})\geqslant \varepsilon.
\end{equation}
Let for each $i$ the integer $p_i$ be the smallest number (possibly $p_i=1$) such that~(\ref{rr1}) holds.
Thus, we may assume that
$$
d(x_{n_i},x_{n_i+p_i-1})<\varepsilon.
$$
Hence, for each $i\in \mathbb{N}$, we have
\begin{align*}
\varepsilon\leqslant d(x_{n_i}, x_{n_i+p_i})\leqslant
&~d(x_{n_i}, x_{n_i+p_i-1}) + d(x_{n_i+p_i-1}, x_{n_i+p_i})\\
< &~\varepsilon + d(x_{n_i+p_i-1}, x_{n_i+p_i}),
\end{align*}
and it follows from asymptotic regularity that $$\lim_{i\rightarrow\infty} d(x_{n_i},x_{n_i+p_i}) =\varepsilon.$$

Note that $p_i>1$ for all $i>N$, where $N$ is a sufficiently large integer. Indeed, otherwise~(\ref{rr1}) contradicts to asymptotic regularity.
Since by assumption $(x_n)$ does not possess a fixed point of $T$, by Remark~\ref{r32}, the points $x_{n_i}$, $x_{n_i+p_i-1}$, $x_{n_i+p_i}$ are pairwise distinct for $i>N$. Using repeated triangle inequality and inequality~(\ref{e5}), we get
\begin{align*}
& {} d(x_{n_{i}},x_{n_i+p_i})\leqslant d(x_{n_{i}},x_{n_i+1})+d(x_{n_i+1},x_{n_i+p_i+1})+d(x_{n_i+p_i+1},x_{n_i+p_i})\\\leqslant & {}
d(x_{n_{i}},x_{n_i+1})+d(x_{n_i+1},x_{n_i+p_i+1})+d(x_{n_i+p_i+1},x_{n_i+p_i})+d(x_{n_i+1},x_{n_i+p_i})\\\leqslant & {}
d(x_{n_{i}},x_{n_i+1})+\varphi \max\{d(x_{n_i},x_{n_i+p_i}), d(x_{n_i+p_i},x_{n_i+p_i-1}), d(x_{n_i},x_{n_i+p_i-1})\}\\ & {}
+F(d(x_{n_i},Tx_{n_i}),d(x_{n_i+p_i},Tx_{n_i+p_i}),d(x_{n_i+p_i-1},Tx_{n_i+p_i-1}))\\ \leqslant & {}
d(x_{n_{i}},x_{n_i+1})+\varphi \max\{d(x_{n_i},x_{n_i+p_i}), d(x_{n_i+p_i},x_{n_i+p_i-1}), d(x_{n_i},x_{n_i+p_i}) \\ & {}+d(x_{n_i+p_i},x_{n_i+p_i-1})\}
+F(d(x_{n_i},Tx_{n_i}),d(x_{n_i+p_i},Tx_{n_i+p_i}),d(x_{n_i+p_i-1},Tx_{n_i+p_i-1})).\\
\end{align*}

Letting $i\rightarrow\infty$, using asymptotic regularity and upper semi-continuity of $\varphi$, we obtain $$0<\varepsilon =\lim_{i\rightarrow\infty} d(x_{n_i},x_{n_i+p_i}) \leqslant
\limsup _{i\rightarrow\infty}\varphi(d(x_{n_i}, x_{n_i+p_i})\leqslant \varphi(\varepsilon)<\varepsilon,$$ which is a contradiction. Hence $(x_n)$ is a Cauchy sequence. The rest of the proof follows from the proof given in Theorem~\ref{gkt2m}.
\end{proof}

\section{Fixed-point theorems in incomplete metric spaces}\label{sec5}

The following theorem is an analogue of Theorem 1 from~\cite{Ka69}. Note that in this theorem, we omit the completeness of the metric space and have two new conditions (iii) and (iv).

\begin{thm}\label{t2}
Let $(X,d)$, $|X|\geqslant 3$, be a metric space and let the mapping $T\colon X\to X$ satisfy the following four conditions:
\begin{itemize}
  \item [(i)] $T$ does not possess periodic points of prime period $2$.
  \item [(ii)] $T$ is a generalized \'CRR type mapping on $X$.
  \item [(iii)] $T$ is continuous at $x^*\in X$.
  \item [(iv)] There exists a point $x_0\in X$ such that the sequence of iterates $x_n=Tx_{n-1}$, $n=1,2,...$, has a subsequence $x_{n_k}$, converging to $x^*$.
\end{itemize}
Then $x^*$ is a fixed point of $T$. The number of fixed points is at most two.
\end{thm}
\begin{proof}
Since $T$ is continuous at $x^*$  and $x_{n_k} \to x^*$ we have $Tx_{n_k}=x_{n_{k}+1} \to Tx^*$. Note that $x_{n_{k}+1}$ is a subsequence of $x_n$ but not obligatory the subsequence of $x_{n_k}$. Suppose $x^*\neq Tx^*$. Consider two balls $B_1=B_1(x^*,r)$ and $B_2=B_2(Tx^*,r)$, where $r<\frac{1}{3}d(x^*, Tx^*)$. Consequently, there exists a positive integer $N$ such that $i>N$ implies
$$
x_{n_i} \in B_1 \, \text{ and } \, x_{n_i+1} \in B_2.
$$
Hence,
\begin{equation}\label{w4}
d(x_{n_i}, x_{n_i+1})>r
\end{equation}
for $i>N$.

If the sequence $x_n$ does not contain a fixed point of the mapping $T$, then we can apply considerations of Theorem~\ref{t1}.
By~(\ref{e9}) for $n=3,4,\ldots$ we have
$$
d(x_{n-1},x_n)\leqslant \gamma^{\frac{n}{2}-1}a,
$$
where $a=\max\{d(x_{0},x_{1}),d(x_{1},x_{2})\}$ and $\gamma=(2(2\alpha+\lambda))/(2-\lambda)\in [0,1)$.
Hence,
$$
d(x_{n_i},x_{n_i+1})\leqslant \gamma^{\frac{n_i+1}{2}-1}a.
$$
But the last expression approaches $0$ as $i\to \infty$ which contradicts to~(\ref{w4}). Hence, $Tx^*=x^*$.

The existence of at most two fixed points follows form the last paragraph of Theorem~\ref{t1}.
\end{proof}

In the following theorem, we suppose that $T$ is a generalized \'CRR type mapping, not defined on the entire space $X$ but on an everywhere dense subset of $X$. Additionally, we assume that $T$ is continuous on $X$ but not only at the point $x^*$. This can be compared with Theorem 2 from~\cite{Ka69}.

\begin{thm}\label{t3}
Let $(X,d)$, $|X|\geqslant 3$, be a metric space and let the mapping $T\colon X\to X$ be continuous. Suppose that
\begin{itemize}
  \item [(i)] $T$ does not possess periodic points of prime period $2$.
  \item [(ii)] $T$ is a generalized \'CRR type mapping on $(M,d)$, where $M$ is an everywhere dense subset of $X$.
  \item [(iii)] There exists a point $x_0\in X$ such that the sequence of iterates $x_n=Tx_{n-1}$, $n=1,2,...$, has a subsequence $x_{n_k}$, converging to $x^*$.
\end{itemize}
Then $x^*$ is a fixed point of $T$. The number of fixed points is at most two.
\end{thm}
\begin{proof}
The proof will follow from Theorem~\ref{t2}, if we can show that $T$ is a generalized \'CRR type mapping on $X$. Let $x, y, z$ be any three pairwise distinct points of $X$ such that  $x, y \in M$, $z\in X \setminus M$ and  let $(c_n)$ be a sequence in $M$ such that $c_n\to z$, $c_n\neq x$, $c_n\neq y$ for all $n$ and $c_i\neq c_j$, $i\neq j$. Then
$$
 d(Tx,Ty)+d(Ty,Tz)+d(Tx,Tz) \leqslant
$$
$$
d(Tx,Ty)+d(Ty,Tc_n)+d(Tc_n,Tz)+d(Tx,Tc_n)+d(Tc_n,Tz)
$$
$$
\leqslant \alpha(d(x,y)+d(y,c_n)+d(c_n,x))
$$
$$
+ \lambda(d(x,Tx)+d(y,Ty)+d(c_n,Tc_n))+2d(Tc_n,Tz)
$$
(using the inequality
\begin{equation}\label{s1}
  d(c_n,Tc_n)\leqslant d(c_n,z)+d(z,Tz)+d(Tz,Tc_n),
\end{equation}
we get)
$$
\leqslant \alpha(d(x,y)+d(y,c_n)+d(c_n,x))
$$
$$
+ \lambda(d(x,Tx)+d(y,Ty)+d(z,Tz))+\lambda d(c_n,z)+\lambda d(Tz,Tc_n)+2d(Tc_n,Tz).
$$
Letting $n\to \infty $ we get $d(c_n,z)\to 0$ and $d(Tc_n,Tz)\to 0$. Hence, inequality~(\ref{e00}) follows.

Let now  $x \in M$, $y, z\in X \setminus M$, and let $(b_n), (c_n)$ be sequences in $M$ such that $b_n\to y$ and $c_n\to z$. (Here and below we consider that the points $x, y, z$ and all elements of sequences converging to these points are pairwise distinct.) Then
$$
 d(Tx,Ty)+d(Ty,Tz)+d(Tx,Tz) \leqslant d(Tx,Tb_n)+d(Tb_n,Ty)
$$
$$
+d(Ty,Tb_n)+d(Tb_n,Tc_n)+d(Tc_n,Tz)+
d(Tx,Tc_n)+d(Tc_n,Tz)
$$
$$
\leqslant \alpha(d(x,b_n)+d(b_n,c_n)+d(c_n,x))
$$
$$
+ \lambda(d(x,Tx)+d(b_n,Tb_n)+d(c_n,Tc_n))+2d(Tb_n,Ty)
+2d(Tc_n,Tz)\leqslant
$$
(using the inequality
\begin{equation}\label{s2}
  d(b_n,Tb_n)\leqslant d(b_n,y)+d(y,Ty)+d(Ty,Tb_n)
\end{equation}
and inequality~(\ref{s1}) we get)
$$
\leqslant \alpha(d(x,b_n)+d(b_n,c_n)+d(c_n,x))
$$
$$
+ \lambda(d(x,Tx)+d(y,Ty)+d(z,Tz))
+2d(Tb_n,Ty)+2d(Tc_n,Tz)
$$
$$
+\lambda(d(b_n,y)+d(Ty,Tb_n)+d(c_n,z)+d(Tz,Tc_n)).
$$
Again, letting $n\to \infty $, we get inequality~(\ref{e00}).

Let now $x, y, z\in X \setminus M$, and let $(a_n), (b_n)$ and $(c_n)$ be sequences in $M$ such that $a_n\to x$, $b_n\to y$ and $c_n\to z$.
Then
$$
d(Tx,Ty)+d(Ty,Tz)+d(Tx,Tz)
$$
$$
\leqslant d(Tx,Ta_n)+d(Ta_n,Tb_n)+d(Tb_n,Ty)
$$
$$
+d(Ty,Tb_n)+d(Tb_n,Tc_n)+d(Tc_n,Tz)
$$
$$
+d(Tx,Ta_n)+d(Ta_n,Tc_n)+d(Tc_n,Tz)
$$
$$
\leqslant \alpha(d(a_n,b_n)+d(b_n,c_n)+d(c_n,a_n))
$$
$$
+ \lambda(d(a_n,Ta_n)+d(b_n,Tb_n)+d(c_n,Tc_n))
$$
$$
+2d(Ta_n,Tx)
+2d(Tb_n,Ty)
+2d(Tc_n,Tz)
$$
(using the inequality
\begin{equation*}
  d(a_n,Ta_n)\leqslant d(a_n,x)+d(x,Tx)+d(Tx,Ta_n)
\end{equation*}
and inequalities~(\ref{s1}) and~(\ref{s2}) we get)
$$
\leqslant \alpha(d(a_n,b_n)+d(b_n,c_n)+d(c_n,a_n))
$$
$$
+ \lambda(d(x,Tx)+d(y,Ty)+d(z,Tz))
$$
$$
+2d(Ta_n,Tx)
+2d(Tb_n,Ty)
+2d(Tc_n,Tz)
$$
$$
+\lambda(
d(a_n,x)+d(Tx,Ta_n)+
d(b_n,y)+d(Ty,Tb_n)+d(c_n,z)+d(Tz,Tc_n)).
$$
Again, letting $n\to \infty$, we get inequality~(\ref{e00}). Hence, $T$ is a generalized \'CRR type mapping on $X$, which completes the proof.
\end{proof}

\bmhead{Acknowledgements}

This work was partially supported by a grant from the Simons Foundation (Award 1160640, Presidential Discretionary-Ukraine Support Grants, E. Petrov).




\end{document}